\documentclass[12pt]{amsart}
\usepackage[margin=1in]{geometry}
\usepackage[utf8]{inputenc}
\usepackage{amsmath}
\usepackage{amssymb}
\usepackage{amsthm}
\usepackage{ytableau}
\usepackage{todonotes}
\usepackage{verbatim}

\title{The Algebra Of Schur Operators}
\author{Ricky Ini Liu}
\address{Department of Mathematics, North Carolina State University, Raleigh, NC}
\email{riliu@ncsu.edu}
\author{Christian Smith}
\address{Department of Mathematics, North Carolina State University, Raleigh, NC}
\email{casmit34@ncsu.edu}
\date{\today}
\thanks{R. I. Liu and C. Smith were partially supported by National Science Foundation grant DMS-1700302.}

\newcommand{\YY}{\mathbf Y}
\newcommand{\CC}{\mathbf C}
\newcommand{\NN}{\mathbf N}
\newcommand{\U}{\mathcal U}

\newcommand{\s}[1]{\mathsf{#1}}

\DeclareMathOperator{\rw}{rw}
\newcommand{\Kequiv}{\stackrel{K}{\sim}}

\newtheorem{theorem}{Theorem}[section]
\newtheorem*{theorem*}{Theorem}
\newtheorem{corollary}[theorem]{Corollary}
\newtheorem{proposition}[theorem]{Proposition}
\newtheorem{lemma}[theorem]{Lemma}
\theoremstyle{definition}

\newtheorem{example}[theorem]{Example}

\begin{document}

\maketitle

\begin{abstract}
We study a representation of the (local) plactic monoid given by Schur operators $u_i$, which act on partitions by adding a box in column $i$ (if possible). In particular, we give a complete list of the relations that hold in the algebra of Schur operators.
\end{abstract}

\section{Introduction}
The \emph{Schur operator} (or \emph{column box-adding operator}) $u_i$ for $i = 1, 2, \dots$ acts on partitions $\lambda$ by adding a box to the $i$th column of the Young diagram of $\lambda$ if the resulting diagram is a partition, otherwise $u_i$ sends $\lambda$ to $0$. These operators were introduced by Fomin in \cite{fomin} and also described by Fomin and Greene in \cite{fomingreene} in their development of the theory of noncommutative Schur functions (which are a useful tool for studying Schur positivity and related phenomena). They can also be thought of as refinements of the box-adding operator $U$ acting on Young's lattice as defined by Stanley \cite{stanley} in his study of differential posets.

The authors of \cite{fomingreene} observe that the Schur operators satisfy the relations of the \emph{local plactic monoid/algebra} with relations:
\begin{align*}
u_iu_j &= u_j u_i \qquad \qquad \text{for $|j - i| \geq 2$}, \\
u_i u_{i+1} u_i &= u_{i+1} u_i u_i, \\
u_{i+1} u_{i+1} u_i &= u_{i+1} u_i u_{i+1}.
\end{align*}
However, they remark that the full set of relations satisfied by the $u_i$ is unknown. In this paper we describe the complete set of relations among the $u_i$ and thereby give a full characterization of the \emph{algebra of Schur operators} by proving the following theorem.

\begin{theorem*}
	The algebra of Schur operators is defined by the relations:
	\begin{align*}
	u_iu_j &= u_j u_i \qquad\qquad \text{ for } |j - i| \geq 2, \\
	u_i u_{i+1} u_i &= u_{i+1} u_i u_i, \\
	u_{i+1} u_{i+1} u_i &= u_{i+1} u_i u_{i+1},  \\
	u_{i+1} u_{i+2} u_{i+1} u_i &= u_{i+1} u_{i +2} u_i u_{i+1}. 
	\end{align*}
\end{theorem*}

%

Interestingly, this algebra is somewhat more complicated than a more common related one, also described in \cite{fomingreene} (see also \cite{billeyjockuschstanley}), that is generated by \emph{diagonal box-adding operators} $\tilde u_i$ that add a box to the $i$th \emph{diagonal} of $\lambda$ if possible (where the diagonals are labeled $1, 2, \dots$ from bottom to top).
The algebra generated by such operators was shown in \cite{billeyjockuschstanley} to be the \emph{nil-Temperley-Lieb algebra} given by the relations:
\begin{align*}
\tilde u_i^2 &= 0, \\
\tilde u_i\tilde u_j &= \tilde u_j \tilde u_i \qquad \qquad \qquad \text{for $|j - i| \geq 2$}, \\
 \tilde u_i \tilde u_{i+1}\tilde u_i &= \tilde u_{i+1} \tilde u_i \tilde u_{i+1} = 0.
\end{align*}


We will begin with some preliminary background in Section $2$ and then move on to a proof of our main theorem in Section $3$.

\section{Preliminaries}

In this section, we will introduce necessary background about partitions, Knuth equivalence, and Schur operators.

\subsection{Partitions}

A \emph{partition} $\lambda = (\lambda_1, \dots, \lambda_n)$ of $|\lambda| = \sum_i \lambda_i$ is a nonincreasing sequence of nonnegative integers. (We may add or delete trailing zeroes as convenient.) To each partition, we associate a \emph{Young diagram}, which is a collection of left aligned boxes with $\lambda_1$ boxes in the first row, $\lambda_2$ boxes in the second row, and so on.  We also define the \emph{conjugate partition} $\lambda'$ to be the partition whose Young diagram is obtained from that of $\lambda$ by reflecting across its main diagonal.

The set of partitions forms a partially ordered set called \emph{Young's lattice} $\YY = (\YY, \subseteq)$, where $\lambda \subseteq \mu$ if and only if the Young diagram of $\lambda$ fits inside the Young diagram of $\mu$ (or equivalently, $\lambda_i \leq \mu_i$ for all $i$).  In this partial order, $\mu$ covers $\lambda$ if and only if $\mu / \lambda$ is a single box.  Here, $\mu / \lambda$ denotes the \emph{skew Young diagram} obtained by deleting those boxes in $\mu$ that are also contained in $\lambda$.

A \emph{semistandard Young tableau} (SSYT) of shape $\lambda$ is formed by filling each box of the Young diagram of $\lambda$ with a positive integer such that the numbers are weakly increasing within a row (read from left to right) and strictly increasing within a column (read from top to bottom).  A \emph{standard Young tableau} (SYT) is a semistandard Young tableau of shape $\lambda$ with labels $1, 2, \dots, |\lambda|$.

The \emph{reading word} $\rw(T)$ of a tableau $T$ is the word obtained by listing the entries of the tableau by rows from bottom to top, reading each row from left to right.  

\ytableausetup{centertableaux}
\begin{example}
Let $\lambda = (4,3,1)$. For the semistandard Young tableau
\[ T = 
\begin{ytableau}
 1 & 1 & 1 & 4 \\
 2 & 2 & 3 \\
 3
\end{ytableau},
\]
we have $\rw(T) = 32231114$.
\end{example}

The \emph{weight} of a tableau $T$ is the tuple $w(T) = (w_1(T), w_2(T), \dots)$, where $w_i(T)$ is the number of occurrences of $i$ in $T$. We similarly define the \emph{weight} $w(x) = (w_1(x), w_2(x), \dots)$ of any word $x$ in the alphabet $\NN = \{1, 2, \dots\}$. (Clearly $T$ and $\rw(T)$ have the same weight.)

\subsection{Schur operators}

Let $\U$ be the free associative algebra (over $\CC$) generated by $u_i$ for $i \in \NN$. Given a word $x = x_1 \cdots x_l$ in the alphabet $\NN$, we define the element $u_x = u_{x_1} \cdots u_{x_l} \in \U$. Hence the set of $u_x$ for all words $x$ forms a basis for $\U$. 

Let $\CC[\YY]$ be the complex vector space with basis $\YY$. Then $\U$ acts on $\CC[\YY]$ as \emph{Schur operators} by
\[u_i(\lambda) =
\begin{cases}
\mu & \text{if $\mu /\lambda$ is a single box in column $i$,} \\
0& \text{otherwise,}
\end{cases}\]
and $u_x(\lambda) = u_{x_1} u_{x_2} \ldots u_{x_l}(\lambda)$,
extended linearly.


\begin{example}
	Let $\lambda = (3,1)$. Then $u_2(\lambda) = (3,2)$, $u_3u_2(\lambda) = (3,3)$, but $u_2u_3u_2(\lambda) = 0$ since adding another box to the second column does not yield a partition.
	\ytableausetup{smalltableaux}
	\[\ydiagram{3,1} \;\xrightarrow{u_2}\; \ydiagram{3,2} \;\xrightarrow{u_3}\; \ydiagram{3,3} \;\xrightarrow{u_2}\; 0\]
\end{example}

Let $I$ be the two-sided ideal of $\U$ consisting of all elements that annihilate all of $\CC[\YY]$. Then two elements $u$ and $u'$ of $\U$ are equivalent modulo $I$, written $u \equiv u' \pmod I$, if $u(\lambda)= u'(\lambda)$ for all partitions $\lambda$.  We call $\U / I$ the \emph{algebra of Schur operators}.


As mentioned in the introduction, the Schur operators were introduced by Fomin \cite{fomin} and discussed by Fomin and Greene \cite{fomingreene} in their study of noncommutative Schur functions.  In particular, they observe that the Schur operators give a representation of the \emph{local plactic monoid}, meaning that the following relations hold modulo $I$:
\begin{align*}
 u_iu_j &\equiv u_j u_i \qquad\qquad \text{for $|j - i| \geq 2$}, \\
 u_i u_{i+1} u_i &\equiv u_{i+1} u_i u_i, \\
 u_{i+1} u_{i+1} u_i &\equiv u_{i+1} u_i u_{i+1}.
\end{align*}
For completeness, we will verify these relations in Section 3 below.

\subsection{Knuth equivalence and RSK}

Consider words $x = x_1 x_2 \dots$, $y = y_1 y_2 \dots$ in the alphabet $\NN = \{1, 2, 3, \dots\}$.  We say that $x$ and $y$ are \emph{Knuth equivalent}, denoted $x \Kequiv y$, if one can be obtained from the other by applying a sequence of \emph{Knuth} or \emph{plactic relations} of the form
\begin{align*}
    \dots bac \dots\;  &\Kequiv \;\dots bca \dots \qquad \text{for }a<b \leq c,\\
    \dots acb \dots\; &\Kequiv \;\dots cab \dots \qquad \text{for } a \leq b < c.
\end{align*}
Here, the ellipses indicate that the subwords occurring before and after the swapped letters remain unchanged. (The Knuth relations define the so-called \emph{plactic monoid} \cite{LascouxSchutzenberger}, of which the local plactic monoid is a quotient.)

The Robinson-Schensted-Knuth (RSK) algorithm gives a bijection between words $x$ and pairs of tableaux $(P,Q)$ where the \emph{insertion tableau} $P$ is semistandard, the \emph{recording tableau} $Q$ is standard, and $P$ and $Q$ have the same shape. (See, for instance, \cite{sagan} for more information.) The exact details of the RSK algorithm will not be important for us, as we will only need the following facts.
\begin{itemize}
	\item The insertion tableau $P$ has the same weight as $x$.
	\item Two words $x$ and $y$ are Knuth equivalent if and only if they have the same insertion tableau $P$.
	\item For any semistandard tableau $P$, the insertion tableau of $\rw(P)$ is $P$.
\end{itemize}

For instance, these facts imply the following proposition, which we will need for our main theorem. (Here and elsewhere, we use $i^k$ to denote a subword of the form $\underbrace{ii \dots i}_k$.)


%

\begin{proposition} \label{splitprop}
	Let $x$ be a word with minimum letter $i$, and let $k=w_i(x)$. Then $x$ is Knuth equivalent to a word $y = \dots i^k \dots$ in which all occurrences of $i$ are consecutive.
\end{proposition}
\begin{proof}
%
%
%
	Let $P$ be the insertion tableau of $x$, and let $y=\rw(P)$. Since $x$ and $y$ both insert to $P$, we have $x \Kequiv y$, and $y$ has the desired form since all $i$'s appear next to each other in the first row of $P$.
\end{proof}

\section{Results}

%
%

Recall that we define $I$ to be the ideal that gives the relations among the Schur operators acting on $\CC[\YY]$. The overall goal of this section is to show that $I$ is generated by the local plactic relations (1)--(3) and one additional type of relation (4) shown below.
\begin{align}
  u_iu_j &\equiv u_j u_i  \qquad\qquad\qquad \text{ for } |j - i| \geq 2, \\
  u_i u_{i+1} u_i &\equiv u_{i+1} u_i u_i, \\
  u_{i+1} u_{i+1} u_i &\equiv u_{i+1} u_i u_{i+1},  \\
  u_{i+1} u_{i+2} u_{i+1} u_i &\equiv u_{i+1} u_{i +2} u_i u_{i+1}. 
\end{align}

\subsection{Equivalence of words}

Our first step is to understand when $u_x \equiv u_y \pmod I$ for $u_x,u_y \in \U$.  To this end we let $\alpha(x) = (\alpha_1(x), \alpha_2(x), \ldots)$, where
\[
\alpha_i(x) = \max \{w_{i+1}(\tilde{x})   - w_i(\tilde{x}) \mid   \text{$\tilde{x}$ is a suffix of $x$} \}.
\] 
(Here, a \emph{suffix} of  $x = x_1x_2 \cdots x_l$ is a trailing subword of the form $\tilde x = x_j x_{j+1} \cdots x_l$, possibly empty.)

\begin{proposition}
\label{whenzero}
Let $\lambda$ be a partition and $x$ a word.  Then 
\[u_x(\lambda) = \begin{cases}
(\lambda_1' + w_1(x), \;\lambda_2' + w_2(x), \;\dots)' & \text{if $\alpha_i(x) \leq \lambda'_i - \lambda'_{i+1}$ for all $i$,}\\
0&\text{otherwise}.
\end{cases}\]
\end{proposition}
\begin{proof}
If $u_x(\lambda) = u_{x_1} \cdots u_{x_l}(\lambda) \neq 0$, then adding boxes to columns $x_l, x_{l-1}, \dots$ of $\lambda$ must always yield a partition. Hence for all $i$ and suffixes $\tilde x$, 
\[\lambda'_{i+1} + w_{i+1}(\tilde x) \leq \lambda'_{i} + w_{i}(\tilde x),\]
so
\[w_{i+1}(\tilde x) - w_{i}(\tilde x) \leq \lambda'_{i} - \lambda'_{i+1}\] for all $i$ and $\tilde x$, which implies $\alpha_i(x) \leq \lambda_i'-\lambda_{i+1}'$.

Otherwise, if $u_x(\lambda) = 0$, then $u_{\tilde x}(\lambda)=0$ for some minimal suffix $\tilde x$, so for some $i$,
\[\alpha_{i}(x) \geq w_{i+1}(\tilde x) - w_{i}(\tilde x) > \lambda'_{i} - \lambda'_{i+1}.\qedhere\]
%
\end{proof}

\begin{corollary}
\label{wordsame}
Let $x = x_1 \ldots x_l$ and $y = y_1 \ldots y_l$ be words. Then $u_x \equiv u_y \pmod I$ if and only if $w(x) = w(y)$ and $\alpha(x) = \alpha(y)$. 
\end{corollary}
\begin{proof}

If $w(x) = w(y)$ and $\alpha(x) = \alpha(y)$, then $u_x \equiv u_y \pmod I$ by Proposition~\ref{whenzero}.

Conversely, if $w(x) \neq w(u)$, then let $\lambda$ be a partition such that
\[
\alpha_i(x), \alpha_i(y) \leq \lambda'_i - \lambda'_{i+1} \text{ for all } i.
\]
Then $u_x(\lambda) \neq u_y(\lambda)$ by Proposition~\ref{whenzero}, which implies $u_x \not\equiv u_y \pmod I$.
If instead  $\alpha(x) \neq \alpha(y)$, then suppose without loss of generality that $\alpha_j(x) < \alpha_j(y)$ for some $j$.  Choose $\lambda$ such that
$
\alpha_i(x) \leq \lambda'_i - \lambda'_{i+1}  \text{ for all } i
$,
but
$
 \alpha_j(y) > \lambda'_j - \lambda'_{j+1}$.
Then $u_x(\lambda) \neq 0 = u_y(\lambda)$, so again $u_x \not\equiv u_y \pmod I$.
\end{proof}

In other words, a word $u_x$ is determined modulo $I$ by $w(x)$ and $\alpha(x)$.
We next verify that $I$ is a binomial ideal, so that Corollary~\ref{wordsame} essentially determines all of the relations in $I$.

\begin{proposition}
	\label{termcor}
	The ideal $I$ is generated by elements of the form $u_x -u_y$ for words $x$ and $y$ such that $\alpha(x) = \alpha(y)$ and $w(x) = w(y)$.
\end{proposition}
\begin{proof}
	Let $I'$ be the ideal of $\U$ generated by $u_x - u_y$ as described above. By Corollary~\ref{wordsame}, we have $I' \subseteq I$. Let $R$ be any element of $I$. Then $R \equiv R' \pmod {I'}$ for some
	\[
	R' = \sum_k c_k u_{x(k)},
	\]
	where for each $k$, $x(k)$ is a word, $0 \neq c_k \in \CC$, and $u_{x(k)} \not \equiv u_{x(k')} \pmod {I'}$ for $k \neq k'$.
	
	Fix some weight $w$ and let $x(k_1), x(k_2), \dots$ be those words in $R'$ for which $w(x(k)) = w$, with $\alpha(x(k_1)), \alpha(x(k_2)), \dots$ ordered lexicographically.  We construct a partition $\lambda$ such that
	\[
	\alpha_i(x(k_1)) = \lambda'_i - \lambda'_{i+1} \text{ for all } i.
	\]
	Proposition~\ref{whenzero} gives $u_{x(k_1)}(\lambda) \neq 0$, but $u_{x(k_j)}(\lambda) = 0$ for all $k_j \neq k_1$ since by the lexicographic ordering, $\alpha_i(x(k_j))>\alpha_i(x(k_1)) = \lambda_i'-\lambda_{i+1}'$ for some $i$.    This then implies that $c_{k_1} = 0$ which is a contradiction unless $R' = 0$.  Thus $R \in I'$ and so $I = I'$.  
\end{proof}

We therefore need only determine relations that allow us to equate $u_x$ for all words $x$ with a fixed $\alpha(x)$ and $w(x)$.

Another useful fact about $I$ is that it satisfies a certain shift invariance.

\begin{corollary} \label{addone}
Let $x = x_1 \dots x_l$ and  $y = y_1 \dots y_l$, and define $x' = (x_1+1) \dots (x_l+1)$ and $y' = (y_1+1)  \dots (y_l+1)$.  Then
\[
u_x \equiv u_y \pmod I \quad \text{if and only if} \quad u_{x'} \equiv u_{y'} \pmod I. 
\]
\end{corollary}
\begin{proof}
Since
\begin{alignat*}{2}
\alpha_i(x') &= \alpha_{i-1}(x), \quad & w_i(x') &= w_{i-1}(x),\\
\alpha_i(y')  &= \alpha_{i-1}(y), \quad & w_i(y') &= w_{i-1}(y),
\end{alignat*}
for all $i$, the result follows by Corollary~\ref{wordsame}.
\end{proof}

For the rest of this section, we will let $J$ denote the ideal generated by relations (1)--(4). We first verify that these relations all lie in $I$.

\begin{proposition}
\label{forwardirection}
The relations (1)--(4) hold in $\U/I$, or equivalently, $J \subseteq I$.
\end{proposition}
\begin{proof}
By Corollary~\ref{addone}, we may take $i=1$. Thus by Corollary~\ref{wordsame}, we need only check that $w(x) = w(y)$ and $\alpha(x) = \alpha(y)$ for the appropriate words on both sides of the relation. This is straightforward: for instance, for relation (4),
\begin{align*}
w(2321) &= w(2312) = (1,2,1),\\
\alpha(2321) &= \alpha(2312) = (1,0).
\end{align*}
The other relations follows similarly.
\end{proof}

In particular, we note the following relationship with Knuth equivalence.

\begin{lemma}
\label{knuthequal}
Let $x$ and $y$ be words such that $x \stackrel{K}{\sim} y$. Then $u_x \equiv u_y \pmod J$.
\end{lemma}
\begin{proof}
	If $x$ and $y$ are related by a Knuth move that switches $a < c$, then $u_x \equiv u_y \pmod I$ by (1) if $|a-c| \geq 2$ or by (2) or (3) if $c = a+1$.
%
\end{proof}

We next demonstrate that Knuth equivalence is sufficient to describe equivalence modulo $I$ for words in two letters $i$ and $i+1$.

\begin{proposition} \label{2letterprop}
Let $x$ and $y$ be words in $i$ and  $i+1$. Then $u_x \equiv u_y \pmod I$  if and only if $u_x \equiv u_y \pmod J$. 
\end{proposition}
\begin{proof}
	We claim that the insertion tableau of $x$ is determined by $w(x)$ and $\alpha(x)$. Indeed, since $P$ is semistandard and contains only $i$'s and $i+1$'s, it has at most two rows, and $i$ can only appear in the first row. Then given $w(x) = w(P)$, all that needs to be determined is the number of $i+1$'s in the first row. Since $x \Kequiv \rw(P)$, Corollary~\ref{wordsame} implies that $\alpha_i(x) = \alpha_i(\rw(P))$. But this is clearly the number of $i+1$'s in the first row of $P$ (since $P$ has at least as many $i$'s in its first row as $i+1$'s in its second row).
%

We can now prove the proposition.  The reverse direction follows from Proposition~\ref{forwardirection}, so suppose $u_x \equiv u_y \pmod I$.  By Corollary~\ref{wordsame}, we have $\alpha(x) = \alpha(y)$ and $w(x) = w(y)$.  Hence $x$ and $y$ must have the same insertion tableau by the above claim, so $x \Kequiv y$.  By Lemma~\ref{knuthequal}, it then follows that $u_x \equiv u_y \pmod J$. 
\end{proof}

Note that we have shown that if our words only contain two consecutive letters, then only relations (2) and (3) are needed to determine equivalence modulo $I$.




\subsection{Key lemmas}

When dealing with three or more letters, we will need to utilize relations (1) and (4). The following two lemmas will show the key contexts in which these relations will be used.

Denote by $x[i, j]$ the subword of $x$ consisting only of the letters $i, i+1, \dots,  j$.  For instance, if $x = 1432212$, then $x[1,2] = 12212$, and $x[2,4] = 43222$.

\begin{lemma}
\label{samelemma}
Let  $x$ and $y$ be words in $1, \dots, n$.  If  $x[1, 2] = y[1, 2]$ and $x[2, n] = y[2, n]$, then $u_x \equiv u_y$ modulo relation (1), that is, they are equivalent up to commutation relations.
\end{lemma}

\begin{proof}
	Note that $x[1,2]$ and $x[2,n]$ must have the same number of occurrences of $2$. Then   
\[
\begin{array}{ccccccccccc}
x[1, 2] = y[1, 2] &=& 1^{n_1} &2 &1^{n_2} &2 &1^{n_3} &2& \dots &2& 1^{n_k}, \\
x[2, n] = y[2, n] &=& m^{(1)} &2 &m^{(2)} &2 &m^{(3)} &2& \dots &2& m^{(k)},
\end{array}
\]
where $m^{(j)}$ is a word in $3, \dots, n$ for all $j = 1, \dots, k$.  Then we must have that
\[
\begin{array}{ccccccccccc}
x &=& m_x^{(1)} &2 &m_x^{(2)} &2 &m_x^{(3)} &2 &\dots &2 & m_y^{(k)}, \\
y &=&  m_y^{(1)} &2 &m_y^{(2)} &2 &m_y^{(3)} &2 &\dots &2 & m_y^{(k)},
\end{array}
\]
where $m_x^{(i)}$ and $m_y^{(i)}$ are both words obtained by shuffling together $1^{n_i}$ and $m^{(i)}$.
%
But
$u_{m_x^{(i)}}$ and $u_{m_y^{(i)}}$ are both equivalent modulo relation (1) to $u_1^{n_i}u_{m^{(i)}}$ and hence to each other. It follows that $u_x$ and $u_y$ are also equivalent modulo relation (1).
%
%
%
\end{proof}

The next lemma shows the key application of relation (4). For ease of notation, we will abbreviate $u_1, u_2, \dots$ by $\s{1}, \s{2}, \dots$.

\begin{lemma}
\label{degreefourlemma}
For any positive integer $k$, we have the relations
\begin{align*}
  u_{i+1} u_{i+2}^k u_{i+1} u_i &\equiv u_{i+1} u_{i +2}^k u_i u_{i+1} \pmod J,\\
    u_{i+1} u_{i+2} u_{i}^k u_{i+1} &\equiv u_{i+2} u_{i +1} u_i^k u_{i+1} \pmod J.
  \end{align*}

\end{lemma}
\begin{proof}
We may assume $i=1$.  Then (using the relations indicated)
\begin{alignat*}{3}
    \s{23}^k\s{12} &\equiv  \s{213}^k\s{2} &\qquad&(1) \\
    &\equiv \s{21323}^{k-1} &&(3) \\
    &\equiv \s{23123}^{k-1} &&(1) \\
    &\equiv \s{23213}^{k-1} &&(4) \\
    &\equiv \s{2323}^{k-1}\s{1} &&(1) \\
    &\equiv \s{23}^k\s{21} &&(3) &.
\end{alignat*}
Similarly,
\begin{alignat*}{3}
    \s{231}^k\s{2} &\equiv \s{21}^k\s{32} &\qquad&(1) \\
    &\equiv \s{1}^{k-1}\s{2132} &&(2) \\
    &\equiv \s{1}^{k-1}\s{2312} &&(1) \\
    &\equiv \s{1}^{k-1}\s{2321} &&(4) \\
    &\equiv \s{1}^{k-1}\s{3221} &&(2) \\
    &\equiv \s{1}^{k-1}\s{3212} &&(3) \\
    &\equiv \s{31}^{k-1}\s{212} &&(1) \\
    &\equiv \s{321}^k\s{2} &&(2).
\end{alignat*}
\end{proof}

\subsection{Main result}

We are now ready to prove our main theorem.

\begin{theorem}
\label{tabletheorem}
For words $x$ and $y$, $u_x \equiv u_y \pmod I$ if and only if $u_x \equiv u_y \pmod J$.
\end{theorem}

\begin{proof}
The reverse direction is proven in Proposition~\ref{forwardirection}, so we need only consider the forward direction. We will induct on $n$, the largest letter appearing in $x$ and $y$. The case $n =1$ is trivial, while the case $n=2$ follows from Proposition~\ref{2letterprop}.

Assume the statement holds for words in letters $1, \dots, n-1$ (and hence for words in any $n-1$ consecutive letters by Corollary~\ref{addone}).  Since $x$ and $y$ have the same number of $2$'s, we can construct a word $z$ in letters $1, \dots, n$ such that $x[2,n] = z[2,n]$ and $y[1,2] = z[1,2]$.  We will then show $u_x \equiv u_z \pmod J$ and $u_y \equiv u_z \pmod J$, which will imply $u_x \equiv u_y \pmod J$.

By assumption we have $u_x \equiv u_y \pmod I$, and so by Corollary~\ref{wordsame},
\begin{align*}
u_{x[1,2]} &\equiv u_{y[1,2]} = u_{z[1,2]} \pmod I,  \\
u_{y[2,n]} &\equiv u_{x[2,n]} = u_{z[2,n]} \pmod I.
\end{align*}
By the inductive hypothesis, we then have
\begin{align}
u_{x[1,2]} &\equiv u_{z[1,2]} \pmod J,\\
u_{y[2,n]} &\equiv u_{z[2,n]}  \pmod J.
\end{align}

We therefore need to show that if $u_{x[1,2]} \equiv u_{z[1,2]} \pmod J$ as in (5) and $x[2,n] = z[2,n]$, then $u_x \equiv u_z \pmod J$, and similarly for $y$ and $z$ as in (6). It suffices to check when the two sides of (5) or (6) differ by a single application of one of the relations (1)--(4).

First suppose the relation $u_m \equiv u_{m'}$ used in (5) involves at most one $u_2$. This will be the case unless we are applying (3) with $i=1$. Note that $m$ may not be a consecutive subword inside $x$ because there may be letters $i > 2$ that occur in between the letters of $m$ in $x$. However, by Lemma~\ref{samelemma}, since there is only one occurrence of $2$ in $m$, we can commute these intervening letters to the left or right to get some $u_{x'}$ equivalent to $u_x$ such that $x'$ has $m$ as a consecutive subword. Replacing $m$ with $m'$ in $x'$ then gives a word $z'$ such that $z'[1,2] = z[1,2]$ and $z'[2,n] = z[2,n]$. Hence 
\[u_x \equiv u_{x'} = \ldots u_m \ldots \equiv \ldots u_{m'} \ldots = u_{z'} \equiv u_z.\]

A similar argument holds if the relation $u_m \equiv u_{m'}$ used in (6) involves at most one $u_2$. This will be the case unless we are applying (2) with $i=2$. Hence it remains to check only these remaining two cases.

Suppose in equivalence (5) we are applying (3) with $i=1$ by replacing $\s{221} \equiv \s{212}$. As above, $221$ and $212$ need not appear consecutively inside $x$ and $z$ since there may be intervening letters $i > 2$. However, we may as above commute any such letters not appearing between the $2$'s to the right to get words $x'$ and $z'$ such that
\begin{align*}
x'&=\dots2m21\dots,\\
z'&=\dots2m12\dots,
\end{align*}
where $m$ is a word in  $3, \dots, n$.

By Proposition~\ref{splitprop}, $m \Kequiv m' 3^k m''$ for some words $m'$ and $m''$  in letters $4, \dots, n$.  Lemma~\ref{knuthequal} then gives $u_m \equiv u_{m'} \s{3}^k u_{m''} \pmod J$.  We then have:
\begin{alignat*}{2}
u_x &\equiv \dots \s{2} u_m \s{2 1} \dots \qquad \qquad \qquad &&(\text{Lemma }\ref{samelemma}) \\ 
&\equiv \dots  \s{2} u_{m'} \s{3}^k  u_{m''} \s{2 1}  \dots &&(\text{Lemma~\ref{knuthequal}})  \\ 
&\equiv \dots u_{m'} \s{2} \s{3}^k \s{2 1}  u_{m''} \dots &&(1) \\ 
&\equiv   \dots u_{m'} \s{2} \s{3}^k \s{1 2} u_{m''} \dots  &&(\text{Lemma~\ref{degreefourlemma}}) \\
&\equiv \dots \s{2}  u_{m'} \s{3}^k  u_{m''} \s{1 2} \dots &&(1) \\ 
&\equiv \dots \s{2} u_m \s{1 2} \dots &&(\text{Lemma~\ref{knuthequal}})  \\ 
&\equiv u_z &&(\text{Lemma~\ref{samelemma}}).
\end{alignat*}

Similarly, if in equivalence (6) we are applying (2) with $i=2$ by replacing $\s{232} \equiv \s{322}$, then we may commute out any 1's not appearing between the 2's to get:
\begin{alignat*}{2}
u_z &\equiv \dots \s{2  3 1}^k \s{2} \dots \qquad &&(\text{Lemma~\ref{samelemma}}) \\
&\equiv \dots \s{3 2 1}^k \s{2} \dots  &&(\text{Lemma~\ref{degreefourlemma}}) \\
&\equiv u_y &&(\text{Lemma~\ref{samelemma}}). 
\end{alignat*}
\end{proof}


\begin{corollary}
	The algebra of Schur operators is defined by the relations:
	\begin{align*}
	u_iu_j &\equiv u_j u_i  \qquad \text{ for } |j - i| \geq 2, \\
	u_i u_{i+1} u_i &\equiv u_{i+1} u_i u_i, \\
	u_{i+1} u_{i+1} u_i &\equiv u_{i+1} u_i u_{i+1},  \\
	u_{i+1} u_{i+2} u_{i+1} u_i &\equiv u_{i+1} u_{i +2} u_i u_{i+1}. 
	\end{align*}
\end{corollary}

\begin{proof}
	By Proposition~\ref{termcor}, $I$ is generated by elements of the form $u_x - u_y$.  Theorem~\ref{tabletheorem} then shows that the above relations generate all such elements.
\end{proof}

\begin{example}
Consider the words 
\begin{alignat*}{3}
x &= 23443231,& \qquad x[1,2] &= 221,& \qquad x[2,4] &= 2344323,\\
y &= 23443132,& \qquad y[1,2] &= 212,& \qquad y[2,4] &= 2344332.
\end{alignat*}
 Using the construction described in Theorem~\ref{tabletheorem}, we consider the word 
\[
 z = 23443123, \qquad z[1,2] = 212 = y[1,2], \qquad z[2,4] = 2344323 = x[2,4].
 \]
Note that $x[1,2]$ and $z[1,2]$ differ by a single application of (3) with $i=1$, but these subwords do not appear consecutively within $x$ or $z$. As in the proof of Theorem~\ref{tabletheorem}, we can rewrite the part of $x$ and $z$ between the $2$'s using the Knuth moves $3443\Kequiv 3434 \Kequiv 4334$ to get the $3$'s in the middle so that we can then use commutations to get a consecutive subword of the form $23^k21$. We then use Lemma~\ref{degreefourlemma}, followed by the reverse of the previous procedure:
\begin{alignat*}{2}
u_x = \s{23443231} 
&\equiv \s{24334231} \qquad\qquad\qquad &&(\text{Lemma~\ref{knuthequal}}) \\ 
&\equiv \s{42332143}    &&(1) \\
&\equiv  \s{42331243} &&(\text{Lemma~\ref{degreefourlemma}}) \\ 
&\equiv \s{24334123} &&(1)  \\ 
&\equiv \s{23443123} = u_z  &&(\text{Lemma~\ref{knuthequal}}).
\end{alignat*}
Since $y[2,4]$ and $z[2,4]$ differ by a relation that only involves a single $2$, we need only use commutations before we can apply the appropriate relation (3):
\begin{alignat*}{2}
u_z = \s{23443123} &\equiv \s{23441323} \qquad \qquad \qquad&&(1)  \\
&\equiv \s{23441332}  &&(3)\\
&\equiv \s{23443132} = u_y &&(1).
\end{alignat*}
\end{example}

\section{Acknowledgments}
The authors would like to thank Sergey Fomin for bringing this problem to their attention and for useful conversations.

\bibliographystyle{plain} 
\bibliography{references}

\end{document}